\documentclass[12pt]{amsart}
\usepackage{amssymb}
\usepackage[margin=1in]{geometry}

\title{Weak containment by restrictions of induced representations}
\author{Matthew Wiersma}
\address{Department of Mathematical and Statistical Sciences, University of Alberta, Edmonton, AB, Canada}
\email{mwiersma@ualberta.ca}

\newtheorem{theorem}{Theorem}[section]
\newtheorem{corollary}[theorem]{Corollary}
\newtheorem{prop}[theorem]{Proposition}
\newtheorem{lemma}[theorem]{Lemma}

\theoremstyle{remark}
\newtheorem{remark}[theorem]{Remark}
\newtheorem{claim}{Claim}

\newenvironment{claimproof}[1]{\par\noindent\emph{Proof of Claim}.\space#1}{\hfill $\blacksquare$}

\theoremstyle{definition}

\newtheorem{example}[theorem]{Example}

\newcommand{\fn}{\!:}
\newcommand{\C}{\mathbb C}
\newcommand{\R}{{\mathbb R}}
\newcommand{\N}{{\mathbb N}}
\newcommand{\Hi}{\mathcal{H}}
\newcommand{\lla}{\left\langle}
\newcommand{\rra}{\right\rangle}
\newcommand{\mc}{\mathcal}

\newcommand{\tn}{\textnormal}

\newcommand{\Z}{\mathbb Z}

\newcommand{\F}{\mathbb{F}}
\newcommand{\supp}{\mathrm{supp}}
\newcommand{\Ind}{\mathrm{Ind}}
\newcommand{\mr}{\mathrm}
\newcommand{\sgn}{\mathrm{sgn}}

\newcommand{\B}{\mc B}

\begin{document}

\begin{abstract}
A {\it QSIN group} is a locally compact group $G$ whose group algebra $\mr L^1(G)$ admits a quasi-central bounded approximate identity. Examples of QSIN groups include every amenable group and every discrete group. It is shown that if $G$ is a QSIN group, $H$ is a closed subgroup of $G$, and $\pi\fn H\to\B(\Hi)$ is a unitary representation of $H$, then $\pi$ is weakly contained in $(\Ind_H^G\pi)|_H$. This provides a powerful tool in studying the C*-algebras of QSIN groups. In particular, it is shown that if $G$ is a QSIN group which contains a copy of $\F_2$ as a closed subgroup, then $\mr C^*(G)$ is not locally reflexive and $\mr C^*_r(G)$ does not admit the local lifting property. Further applications are drawn to the ``(weak) extendability'' of Fourier spaces $\mr A_\pi$ and Fourier-Stieltjes spaces $\mr B_\pi$.
\end{abstract}

\maketitle

\section{Introduction}

Let $G$ be a locally compact group, and $\pi\fn G\to \mc B(\Hi_\pi)$ and $\sigma\fn G\to \mc B(\Hi_\sigma)$ be (continuous unitary) representations of $G$. There are three related notions for what it means for $\pi$ to contain the representation $\sigma$. The most basic of these is when $\sigma$ is unitarily equivalent to a subrepresentation of $\pi$, but this notion is too strong for many purposes from the perspective of operator algebras. The other two notions of containment are quasi-containment and weak containment. These are the appropriate versions of containment to consider from the perspectives of von Neumann algebras and C*-algebras, respectively. Indeed, if we define $\mr{VN}_\pi:=\pi(G)''\subset \mc B(\Hi_\pi)$ and $\mr C^*_\pi:=\overline{\pi(\mr L^1(G))}^{\|\cdot\|}\subset \mc B(\Hi_\pi)$, then the identity map on $G$ extends to a (normal) $*$-homomorphism $\mr{VN}_\pi\to\mr{VN}_\sigma$ if and only if $\pi$ quasi-contains $\sigma$, and the identity map on $\mr L^1(G)$ extends to a $*$-homomorphism $\mr C^*_\pi\to \mr C^*_\sigma$ if and only if $\pi$ weakly contains $\sigma$.

Suppose that $H$ is a closed subgroup of a locally compact group $G$, and $\pi\fn H\to \mc B(\Hi_\pi)$ is a representation of $H$. The main result in this paper addresses the question of when the restriction of the induced representation $(\Ind_H^G\pi)|_H$ weakly contains $\pi$. When $G$ is a discrete group, it is a straightforward exercise to check that $\pi$ is unitarily equivalent to a subrepresentation of $(\Ind_H^G\pi) |_H$ for every subgroup $H$ of $G$ and representation $\pi$ of $H$. In 1979 Cowling and Rodway showed that the restriction of the Fourier-Stieltjes algebra of $G$ to $H$, $\mr B(G)|_H$, is equal to $\mr B(H)$ whenever $G$ is a SIN group (see \cite[Theorem 2]{cr}). An examination of their proof reveals that Cowling and Rodway actually established the following result, from which their previously stated theorem follows immediately.
\begin{theorem}[Cowling-Rodway \cite{cr}]\label{CR}
	Let $G$ be a SIN group and $H$ a closed subgroup of $G$. If $\pi\fn H\to \B(\Hi_\pi)$ is a representation of $H$, then $\pi$ is quasi-contained in $(\Ind_H^G\pi)|_H$.
\end{theorem}

A QSIN group is a locally compact group whose group algebra $\mr L^1(G)$ admits a quasi-central bounded approximate identity. This class of groups contains all amenable groups in addition to every SIN group. It is known that Theorem \ref{CR} fails to hold when the class of SIN groups is replaced with QSIN groups; however, the main result of this paper is that the analogue of Theorem \ref{CR} where quasi-containment is replaced with weak containments is true for the class of QSIN groups.




\begin{theorem}\label{intro main}
	Let $G$ be a QSIN group and $H$ a closed subgroup of $G$. If $\pi\fn H\to \B(\Hi_\pi)$ is a representation of $H$, then $\pi$ is weakly contained in $(\Ind_H^G\pi)|_H$.
\end{theorem}
\noindent This result provides a powerful tool for understanding local and approximation properties group C*-algebras of non-discrete groups.

Recall that Lance showed in 1973 that a discrete group $G$ is amenable if and only if $\mr C_r^*(G)$ is nuclear (see \cite{lance}). Since the quotient of a nuclear C*-algebra is nuclear, we also get the characterization that a discrete group $G$ is amenable if and only if $\mr C^*(G)$ is nuclear. Partially motivated by these characterizations, there is a lot of research centred around understanding the local and approximation properties of groups C*-algebras for nonamenable discrete groups.
An early result in the subject due to Wasserman is that the sequence
	$$ 0\to \mr C^*(\F_2)\otimes_{\min}J\to \mr C^*(\F_2)\otimes_{\min}\mr C^*(\F_2)\to \mr C^*(\F_2)\otimes_{\min} C^*_r(\F_2) \to 0$$
is not exact, where $J$ is the kernel of the canonical map $\mr C^*(\F_2)\to\mr C^*_r(\F_2)$ (see \cite{wa}).
It follows that $\mr C^*(\F_2)$ is not locally reflexive and $\mr C^*_r(\F_2)$ does not have the local lifting property.

Lance's characterization of amenability does not hold for nondiscrete groups since $\mr C^*(G)$ is nuclear for every separable connected group $G$ (\cite[Corollary 6.9]{con}). As such, the local and approximation properties of nondiscrete nonamenable groups are considered difficult to understand in general. Theorem \ref{intro main} provides a powerful tool for understanding some of these properties for group C*-algebras of QSIN groups. As a particular application of Theorem \ref{intro main}, it is shown that the analogue of Wasserman's result holds for every QSIN group containing a copy of $\F_2$ as a closed subgroup.

\begin{theorem}
	Let $G$ be a QSIN group which contains $\F_2$ as a closed subgroup. Then the sequence
	\begin{equation*}
	0\to \mr C^*(G)\otimes_{\min}K\to \mr C^*(G)\otimes_{\min}\mr C^*(G)\to \mr C^*(G)\otimes_{\min} C^*_r(G) \to 0
	\end{equation*}
	is not exact, where $K$ is the kernel of the canonical map $\mr C^*(G)\to \mr C_r^*(G)$.
\end{theorem}

\noindent Consequently $\mr C^*(G)$ is not locally reflexive and $\mr C^*_r(G)$ does not admit the LLP for such groups $G$.

In the final section of the paper, Theorem \ref{CR} and Theorem \ref{intro main} are applied to shed light on the problem of when it is possible to ``extend'' a Fourier space $\mr A_\pi$ or ``weakly extend'' a Fourier-Stieltjes space $\mr B_\pi$ from a closed subgroup $H$ of a locally compact group $G$ to a Fourier space or Fourier-Stieltjes space of $G$.




\section{Background and Notation}

Due to the broad scope of topics touched on by this paper, we begin with a relatively extensive background section.


\subsection{Quasi-containment and weak-containment of representations}\label{containment}

Let $G$ be a locally compact group. A representation $\pi\fn G\to \B(\Hi_\pi)$ is {\it quasi-contained} in second representation $\sigma\fn G\to \B(\Hi_\sigma)$ of $G$  if $\pi$ is unitarily equivalent to a subrepresentation of some amplification 
of $\sigma$. Quasi-equivalence is extremely important in the study of von Neumann algebras associated to $G$ since the map $\sigma(g)\mapsto \pi(g)$ for $g\in G$ extends to a normal $*$-homomorphism from $\mr{VN}_\sigma:=\sigma(G)''\subset \B(\Hi_\sigma)$ to $\mr{VN}_\pi:=\pi(G)''\subset \mc B(\Hi_\pi)$ if and only if $\pi$ is quasi-contained in $\sigma$.

Weak containment is essential to the study of group C*-algebras in a similar way as quasi-containment is to group von Neumann algebras. For a representation $\pi\fn G\to \B(\Hi_\pi)$ of $G$ and $\xi,\eta\in\Hi_\pi$, we will let $\pi_{\xi,\eta}\fn G\to \C$ be the matrix coefficient defined by $\pi_{\xi,\eta}(g)=\lla \rho(g)\xi,\eta\rra$. A representation $\pi\fn G\to \B(\Hi_\pi)$ is {\it weakly contained} in $\sigma\fn G\to \B(\Hi_\sigma)$ if every positive definite function of the form $\pi_{\xi,\xi}$ for $\xi\in\Hi_\pi$ is the limit of positive definite functions with the form $\sum_{j=1}^N \sigma_{\eta_j,\eta_j}$ for $N\in\N$ and $\eta_1,\ldots,\eta_N\in \Hi_\sigma$ in the topology of uniform convergence on compact subsets of $G$. The representation $\pi$ is weakly contained in $\sigma$ if and only if $\|\pi(f)\|\leq \|\sigma(f)\|$ for every $f\in \mr L^1(G)$ if and only if the map $\pi(f)\mapsto \sigma(f)$ for $f\in \mr L^1(G)$ extends to a  $*$-homomorphism from $\mr C^*_\sigma:=\overline{\sigma(\mr L^1(G))}^{\|\cdot\|}$ to $\mr C^*_\pi :=\overline{\pi(\mr L^1(G))}^{\|\cdot\|}$.

Further details on quasi-containment and weak containment can be found in \cite{dix}.

\subsection{Some classes of locally compact groups}\label{classes}

A {\it SIN} group (standing for {\it small invariant neighbourhood} group) is a locally compact group $G$ whose identity admits a neighbourhood base of compact sets $K$ which are invariant under conjugation, i.e., such that $s^{-1}Ks=K$ for all $s\in G$. Examples of SIN groups include all discrete, abelian, and compact groups. 

Recall that if $G$ is a locally compact group with modular function $\Delta_G$ and $s\in G$, then $\tau(s)\fn \mr L^1(G)\to \mr L^1(G)$ defined by $\tau(s)f(t)=f(s^{-1}ts)\Delta_G(s)$ is an isometric isomorphism. It is a well known result of Mosak that SIN groups are exactly the class of locally compact groups $G$ for which $\mr L^1(G)$ admits a {\it central bounded approximate identity}, i.e. a bounded approximate identity $\{e_\alpha\}\subset \mr L^1(G)$ such that $\tau(s)e_\alpha=e_\alpha$ for every $s\in G$ and index $\alpha$ (see \cite{mos}). Using this characterization, QSIN groups (standing for {\it quasi-SIN} groups) offer a natural generalization of SIN groups.

A locally compact group $G$ is {\it QSIN} if $\mr L^1(G)$ admits a {\it quasi-central bounded approximate identity}, i.e., a bounded approximate identity $\{e_\alpha\}\subset \mr L^1(G)$ such that $\|\tau(s)e_\alpha-e_\alpha\|_1\to 0$ uniformly on compact subsets of $G$. QSIN groups are much more general than the class of SIN groups and, by a result of Losert and Rindler, contains every amenable group (see \cite{lr}). Interested readers are encouraged to see \cite{stokke} where much of the basic theory of QSIN groups is developed and \cite[Remark 2.2]{lss} for a list of examples and non-examples of QSIN groups.

\subsection{Induced Representations}

We refer the reader to \cite{kt} as a resource on induced representations. We will be following the conventions used within this book, but pause to fix notation and summarize key results.

Let $G$ be a locally compact group and $H$ be a closed subgroup of $G$. If $\pi\fn H\to \mc B(\Hi_\pi)$ is a representation of $H$, we set $\mc F(G,\pi)$ to be the set of norm continuous functions $f \fn G\to \Hi_\pi$ so that
\begin{enumerate}
	\item $f(gh)=\delta_H^G(h)\pi(h^{-1})f(g)$ for $g\in G$ and $h\in H$, and
	\item $q(\supp f)$ is compact in $G/H$
\end{enumerate}
where $\delta_H^G(h):=\Delta_H(h)^{\frac{1}{2}}\Delta_G(h)^{-\frac{1}{2}}$ and $q\fn G\to G/H$ is the quotient map. We define an inner product on $\mc F(G,\pi)$ by letting
$$\lla f_1,f_2\rra=\int_{G}\psi(g)\lla f_1(g),f_2(g)\rra\,dg$$
for $f_1,f_2\in \mc F(G,\pi)$, where $\psi$ is any function in $\mr C_{\mr c}(G)$ such that $\psi^\#(\dot g)=1$ for every $g\in (\supp\,f)\cup(\supp\,g)$, where $\psi^\#\in {\mr C}_{\mr c}(G/H)$ is defined by $\psi^\#(\dot g)=\int_H \psi(gh)\,dh$. Let $\overline{\mc F}(G,\pi)$ denote the completion of $\mc F(G,\pi)$ with respect to this inner product. The induced representation of $\pi$ is given by $\Ind_H^G \pi(s) f=L_s f$ for $f\in \mc F$, where $L_s$ is the continuous extension of the left shift operator $L_sf(g)=f(s^{-1}g)$ from $\mc F(G,\pi)$ to $\overline{\mc F}(G,\pi)$.

Fix a continuous rho-function $\rho$ for $(G,H)$, i.e., a continuous function $\rho\fn G\to (0,\infty)$ such that $\rho(gh)=\delta_H^G(h)^2\rho(g)$ for all $h\in H$, $g\in G$, and a quasi-invariant measure $d\dot g$ on $G/H$ such that
$$ \int_{G/H}\int_H f(gh)\,dh\,d\dot g=\int_G f(g)\rho(g)\,dg$$
for all $f\in \mr C_{\mr c}(G)$. Then a straightforward calculation shows that
$$\lla f,g\rra =\int_{G/H}\frac{1}{\rho(g)}\lla f_1(g),f_2(g)\rra\,d\dot g.$$
We will favour this latter equation in our calculations.

\subsection{Fourier and Fourier-Stieltjes spaces}

The basic theory of Fourier and Fourier-Stieltjes was developed by Eymard in \cite{eymard} and further refined by his student Arsac in \cite{arsac}. We will give an overview of the results from these two works as they apply to this paper.

Let $G$ be a locally compact group. Recall from Section \ref{containment} that if $\pi\fn G\to \mc B(\Hi_\pi)$ is a representation of $G$ and $\xi,\eta\in\Hi_\pi$, then $\pi_{\xi,\eta}\fn G\to \C$ is defined by $\pi_{\xi,\eta}(s)=\lla \pi(s)\xi,\eta\rra$. The {\it Fourier-Stieltjes algebra} $\mr B(G)$ is the set of all matrix coefficients $\pi_{\xi,\eta}$ as $\pi$ ranges over all representations  of $G$ and $\xi,\eta$ over the associated Hilbert spaces $\Hi_\pi$. The Fourier-Stieltjes algebra is a Banach algebra which is the dual of $\mr C^*(G)$ with respect to pointwise operations, norm given by
$$ \|u\|:=\inf\{\|\xi\|\|\eta\| : u=\pi_{\xi,\eta}\},$$
and dual pairing
$$<\! u,f\!>:=\int_G u(s)f(s)\,ds$$
for $u\in B(G)$ and $f\in \mr L^1(G)$.

For each representation $\pi\fn G\to \mc B(\Hi_\pi)$ of $G$, the Fourier space $\mr A_\pi$ is defined to be the norm closed linear span of matrix coefficients $\pi_{\xi,\eta}$ in $\mr B(G)$ as $\xi$ and $\eta$ range over $\Hi_\pi$, and the Fourier-Stieltjes space $\mr B_\pi$ is defined to be the weak* closure of $\mr A_\pi$ in $\mr B(G)$. Then $\mr A_\pi$ can be naturally identified with the predual of $\mr{VN}_\pi$ and $\mr B_\pi$ with the dual of $\mr C^*_\pi$.  Every Fourier space $\mr A_\pi$ is invariant under both left and right translation by $G$ and, conversely, every norm closed subspace $E$ of $\mr B(G)$ which is closed under left and right translation by $G$ is a Fourier space $\mr A_\pi$. As a distinguished Fourier space, the {\it Fourier algebra}  $\mr A(G)$ is defined to be $\mr A_\lambda$, where $\lambda\fn G\to \mc B(\mr L^2(G))$ is the left regular representation of $G$.

Let $\pi$ and $\sigma$ be representations of a locally compact group $G$. Then $\mr A_\sigma\subset \mr A_\pi$ if and only if $\sigma$ is quasi-contained in $\pi$ and $\mr B_\sigma\subset \mr B_\pi$ if and only if $\sigma$ is weakly contained in $\pi$. The Fourier space $\mr A_\pi$ is a subalgebra of $\mr B(G)$ if and only if $\pi\otimes\pi$ is quasi-contained in $\pi$ and, similarly, $\mr B_\pi$ is a subalgebra of $\mr B(G)$ if and only if $\pi\otimes\pi$ is weakly contained in $\pi$. Further, $\mr A_\pi$ is an ideal of $\mr B(G)$ if and only if $\pi\otimes \sigma$ is quasi-contained in $\pi$ for every representation $\sigma$ of $G$, and $\mr B_\pi$ is an ideal of $\mr B(G)$ if and only if $\pi\otimes \sigma$ is weakly contained in $\pi$ for every representation $\sigma$ of $G$. In particular, $\mr A(G)$ is an ideal of $\mr B(G)$ by Fell's absorption principle.

The final facts which we will need to know about Fourier and Fourier-Stieltjes space are the following. Suppose that $H$ is a closed subgroup of a locally compact group $G$ and $\pi$ is a representation of $G$. Then $\mr A_\pi|_{H}=\mr A_{\pi|_H} $ and $\mr B_\pi|_{H}\subset \mr B_{\pi|_H}$.

\subsection{Some local properties of C*-algebras}

A C*-algebra (or an operator space) $A$ is {\it locally reflexive} if for every finite dimensional subspace $E\subset X$ and complete contraction $\phi\fn E\to A^{**}$, there exists a net $\phi_i\fn E\to A$ of complete contractions so that $\phi_i\to \phi$ in the point-weak* topology. Local reflexivity is also equivalent to property C$''$ (see \cite[Proposition 18.15]{pisier}). Recall that a C*-algebra $A$ is {\it exact} if the sequence
$$ 0\to A\otimes_{\min} J\to A\otimes_{\min} B\to A\otimes_{\min} C\to 0$$
is exact for every short exact sequence $0\to J\to B\to C\to 0$ of C*-algebras. Every exact C*-algebra is locally reflexive.

Now suppose that $A$ is a unital C*-algebra. The C*-algebra $A$ has the {\it local lifting property} ({\it LLP}) if for every finite dimensional operator system $E\subset A$, unital C*-algebra $B$ with two-sided closed ideal $J$, and unital completely positive (ucp) map $\phi\fn A\to (B/J)$ there exists a ucp map $\psi\fn E\to B$ such that $\phi|_E=\pi\circ \psi$ where $\pi\fn B\to B/J$ is the quotient map. We say that a non-unital C*-algebra $A$ has the LLP whenever its unitization does. Kirchberg showed that a C*-algebra $A$ has the LLP if and only if
$$ A\otimes_{\min} \mc B(\Hi)=A\otimes_{\max} \mc B(\Hi)$$
canoncially, where $\Hi$ is the separable infinite dimensional Hilbert space (see \cite{ki}).

All of this and further information on local reflexivity and the local lifting property can be found in each of the books \cite{bo,er,pisier}.

\section{Main result}

\begin{theorem}\label{main thm}
	Suppose $G$ is a QSIN group and $H$ is a closed subgroup of $G$. Then
	$ (\Ind_H^G \pi )|_H$ weakly contains $\pi$ for every representation $\pi\fn H\to \mc B(\Hi_\pi)$ of $H$.
\end{theorem}

\begin{proof}
	Fix a unit vector $\xi\in\Hi_\pi$ and recall from Section \ref{classes} that $\tau(s)\fn \mr L^1(G)\to \mr L^1(G)$ is defined by $\tau(s)f(t)=f(s^{-1}ts)\Delta_G(s)$ for every $s\in G$ and $f\in \mr L^1(G)$. Since $G$ is QSIN, there exists a net $\{e_\alpha\}$ in $\mr L^1(G)_1^+\cap \mr C_{\mr c}(G)$ such that $\supp\,e_\alpha\to \{e\}$ and $\|\tau(s)e_\alpha-e_\alpha\|_1\to 0$ uniformly on compact subsets of $G$ (see \cite[Theorem 2.6]{stokke}). For each index $\alpha$, we define $f_{\alpha}\in \mr C(G,\Hi_\pi)$ by
	$$ f_{\alpha}(g)=\int_H \delta_H^G(h^{-1})^2e_{\alpha}(gh)\pi(h)\xi\,dh $$
	and $x_{\alpha}\in \mr C(G,\Hi_\pi)$ by
	$$x_{\alpha}(\dot g)=\sgn(f_{\alpha}(\dot g))\|f_{\alpha}(\dot g)\|^{1/2},$$
	where $\sgn\fn \Hi_\pi\to \Hi_\pi$ is given by
	$$\sgn(\eta)=\left\{\begin{array}{c l}
	\frac{\eta}{\|\eta\|},&\tn{if }\eta\neq 0\\
	0, &\tn{if }\eta =0.
	\end{array}\right.$$
	Observe that $x_\alpha\in \mc F(G,\Hi_\pi)$ since 
	$$ f_\alpha(gh)=\int_H \delta_H^G(h_1^{-1}h)^2e_\alpha(gh_1)\,\pi(h^{-1}h_1)\xi\,dh_1=\delta_H^G(h)^2 \pi(h^{-1})f_{\alpha}(g).$$
	We will show that $$(\Ind_H^G\pi)_{x_\alpha,x_\alpha}(h)\to \pi_{\xi,\xi}(h)$$
	uniformly on compact subsets of $H$.


	\begin{claim}\label{sgn claim}
		Let $K$ be a compact subset of $H$ and $\epsilon>0$. There exists an index $\alpha_0$ so that if $s\in K$, $g\in G$, and $\alpha\geq \alpha_0$, then
			$$ \big|\lla \sgn\,x_{\alpha}(s^{-1}g),\sgn\,x_{\alpha}(g)\rra-\pi_{\xi,\xi}(s)\big|<\epsilon$$
		whenever $x_\alpha(g)\neq 0$ and $x_\alpha(s^{-1}g)\neq 0$.
	\end{claim}

	\begin{claimproof}
	Without loss of generality, we may assume that $\epsilon<2\sqrt{2}$ and $e\in K$.
	Since $\pi(K)\xi$ is a compact subset of $\Hi_\pi$, there exists a symmetric neighbourhood $U$ of the identity in $G$ so that
	$$ \|\pi(h)\pi(s)\xi-\pi(s)\xi\|<\frac{\epsilon}{2}$$
	for all $h\in U^2\cap H$ and $s\in K$. Further, since $K$ is compact and $\supp\,e_\alpha\to \{e\}$, there exists an index $\alpha_0$ so that $\supp\big(\tau(s)\,e_\alpha\big)\subset U$ for all $s\in K$ and $\alpha\geq \alpha_0$.
	We will show that this index $\alpha_0$ produces the desired result.
	
	
	Let $\alpha\geq \alpha_0$. A routine calculation shows that
	$$ f_\alpha(s^{-1}g)=\int_{H}\delta_H^G(h^{-1})^2\tau(s)e_\alpha(gh)\pi(h)\pi(s)\xi\,dh.$$
	If $\tau(s)e_\alpha(gh)\neq 0$ for some $h\in H$, then $g=g_1h_1$ for some $g_1\in U$ and $h_1\in H$. Hence,
	\begin{eqnarray*}
	f_{\alpha}(s^{-1}g) &=& \delta_H^G(h_1)^2\pi(h_1^{-1})f(s^{-1}g_1)\\
	&=&\delta_H^G(h_1)^2\pi(h_1^{-1})\int_{U^2\cap H}\delta_H^G(h^{-1})^2\tau(s)e_\alpha(g_1h)\pi(h)\pi(s)\xi\,dh.
	\end{eqnarray*}
	Since $ \|\pi(h)\pi(s)\xi-\pi(s)\xi\|<\frac{\epsilon}{2}<\sqrt{2} $ for every $h\in U^2\cap H$,
	$$ \|\sgn\,f_\alpha(s^{-1}g)-\pi(h_1)\pi(s)\xi\|<\frac{\epsilon}{2} $$
	for all $s\in K$ which satisfy $f_{\alpha}(s^{-1}g)\neq 0$. Thus,
	$$ \big|\lla \sgn\,x_{\alpha}(s^{-1}g),\sgn\,x_{\alpha}(g)\rra-\pi_{\xi,\xi}(s)\big|<\epsilon$$
	for all $s\in K$ and $g\in G$ which satisfy $x_{\alpha}(g)\neq 0$ and $x_\alpha(s^{-1}g)\neq 0$.
	\end{claimproof}

	\begin{claim} \label{norm claim}
		Let $K$ be a compact subset of $H$. Then
		\begin{itemize}
			\item[(a)] $$\|x_\alpha\|=\left(\int_{G/H} \frac{1}{\rho(g)}\|x_\alpha(g)\|^2\,d\dot g\right)^{1/2}\to 1.$$
			\item[(b)] $$\int_{G/H}\frac{1}{\rho(g)}\big| \|x_{\alpha}(s^{-1}g)\|-\|x_\alpha(g)\|\big|^2\,d\dot g\to 0$$ uniformly for $s\in K$.
		\end{itemize}
	\end{claim}

\begin{claimproof}
	Without loss of generality, we may assume that $e\in K$. Choose a symmetric neighbourhood $U$ of the identity in $G$ so that
	$$\|\pi(h)\pi(s)\xi-\pi(s)\xi\|<\frac{\epsilon}{3}$$
	for all $s\in K$ and $h\in U^2\cap H$, and an index $\alpha_0$ so that $\supp\,\tau(s)e_\alpha\subset U$ for all $s\in K$ and $\alpha\geq \alpha_0$.
	
	Fix $\alpha\geq \alpha_0$ and $g\in G$. If $g=g_1h_1$ for some $g_1\in U$ and $h_1\in H$, then
	\begin{eqnarray*}
	&&\left\|\pi(h_1)f_\alpha(s^{-1}g)-\int_{H}\delta_H^G(h^{-1})^2\tau(s)e_\alpha(gh)\pi(s)\xi\,dh\right\| \\
	&=& \left\|\delta_H^G(h_1)^2 f_\alpha(s^{-1}g_1)-\delta_H^G(h_1^2)\int_{H}\delta_H^G(h^{-1})^2\tau(s)e_\alpha(g_1h)\pi(s)\xi\,dh\right\|\\
	&=& \left\|\delta_H^G(h_1^2)\int_{U^2\cap H}\delta_H^G(h^{-1})^2\tau(s)e_\alpha(g_1h)\big(\pi(h)\pi(s)\xi-\pi(s)\xi\big)\,dh\right\|\\
	&\leq&\frac{\epsilon}{3}\cdot\delta_H^G(h_1)^2\int_{H}\delta_H^G(h^{-1})^2\tau(s)e_\alpha(g_1h)\,dh\\
	&=& \frac{\epsilon}{3}\int_{H}\delta_H^G(h^{-1})^2\tau(s)e_\alpha(gh)\,dh.
	\end{eqnarray*}
	In particular, this implies that
	\begin{equation}\label{1}
	\left| \|f_\alpha(s^{-1}g)\|-\int_H \delta_H^G(h^{-1})^2\tau(s)e_\alpha(gh)\,dh\right|<\frac{\epsilon}{3}\int_H \delta_H^G(h^{-1})^2\tau(s)e_\alpha(gh)\,dh.
	\end{equation}
	Since
	$$\int_{G/H}\frac{1}{\rho(g)}\int_H\delta_H^G(h^{-1})^2\tau(s)e_\alpha(gh)\,dh\,d\dot g=1,$$
	it follows immediately that
	$$ \|x_\alpha\|^2=\int_{G/H}\frac{1}{\rho(g)}\|f_\alpha(s^{-1}g)\|\,dg\to 1. $$
	This establishes part (a) of the claim.

	Next observe that
	\begin{eqnarray*}
	&&\int_{G/H}\frac{1}{\rho(g)}\left|\int_{H}\delta_H^G(h^{-1})^2\tau(s)e_\alpha(gh)\,dh-\int_H\delta_H^G(h^{-1})^2e_\alpha(gh)\,dh\right|d\dot g\\
	&\leq& \int_{G/H}\frac{1}{\rho(g)}\int_H \delta_H^G(h^{-1})^2\big|\tau(s)e_\alpha(gh)-e_\alpha(g)\big|\,dh\,d\dot g\\
	&=& \|\tau(s)e_\alpha-e_\alpha\|_1
	\end{eqnarray*}
	Since $\|\tau(s)e_\alpha-e_\alpha\|_1\to 0$ uniformly for $s\in K$, by replacing $\alpha_0$ with a higher index if necessary, we may assume that $\|\tau(s)e_\alpha-e_\alpha\|_1<\frac{\epsilon}{3}$ for all $s\in K$ and $\alpha\geq \alpha_0$. Combining this with Equation \ref{1} gives
	$$ \int_{G/H}\frac{1}{\rho(g)}\big|\|f_\alpha(s^{-1}g)\|-\|f_\alpha(g)\|\big|\,d\dot g\leq\epsilon.$$
	Hence,
	$$\int_{G/H} \frac{1}{\rho(g)}\big|\|x_\alpha(s^{-1}g)\|-\|x_\alpha(g)\|\big|^2\,d\dot g\leq \int_{G/H} \frac{1}{\rho(g)}\big|\|x_\alpha(s^{-1}g)\|^2-\|x_\alpha(g)\|^2\big|\,d\dot g\to 0$$
	uniformly for $s\in K$. This establishes part (b) of the claim.
\end{claimproof}	
	
	\begin{claim}
		Let $K$ be a compact subset of $H$. Then $(\Ind_H^G\pi)_{x_\alpha,x_\alpha}(h)\to \pi_{\xi,\xi}(s)$ uniformly for $s\in K$.
	\end{claim}
\begin{claimproof}
	We begin by observing that if $s\in H$, then
	\begin{eqnarray*}
	(\Ind_H^G\pi)_{x_{\alpha},x_\alpha}(s) &=& \int_{G/H} \frac{1}{\rho(g)}\lla x_{\alpha}(s^{-1}g),x_{\alpha}(g)\rra d\dot g\\
	&=&\int_{G/H} \frac{1}{\rho(g)}\|x_\alpha(s^{-1}g)\|\|x_\alpha(g)\|\lla \sgn\, x_\alpha(s^{-1}g),\sgn\,x_\alpha(g)\rra\,d\dot g.
	\end{eqnarray*}
	By applying Claims \ref{sgn claim} and \ref{norm claim} to this equation, we conclude that $(\Ind_H^G\pi)_{x_\alpha,x_\alpha}(h)\to \pi_{\xi,\xi}(s)$ uniformly for $s\in K$.
\end{claimproof}\\
\end{proof}


\section{Local properties of group C*-algebras}

In 1976 Wasserman gave the first example of a C*-algebra which is not exact.
\begin{theorem}[Wasserman \cite{wa}]\label{wass}
	The sequence
	$$ 0\to \mr C^*(\F_2)\otimes_{\min}J\to \mr C^*(\F_2)\otimes_{\min}\mr C^*(\F_2)\to \mr C^*(\F_2)\otimes_{\min} C^*_r(\F_2) \to 0$$
	is not exact, where $J$ is the kernel of the canonical map $\mr C^*(\F_2)\to\mr C^*_r(\F_2)$.
\end{theorem}

\noindent In 1985 Effros and Haagerup proved the following two results in \cite{eh}.

\begin{theorem}[Effros-Haagerup {\cite[Proposition 5.3]{eh}}]

If $A$ is a locally reflexive C*-algebra, then the sequence
$$0 \to J \otimes_{\min} C \to A \otimes_{\min} C \to A/J \otimes_{\min}  C \to 0$$
is exact for every closed two-sided ideal $J$ of $A$ and every C*-algebra $C$.
\end{theorem}
\begin{theorem}[Effros-Haagerup {\cite[Theorem 3.2]{eh}}]
Let $B$ be a C*-algebra and $J$ a closed two sided ideal of $B$. If $A:=B/J$ has the local lifting property, then the sequence
$$0 \to  J \otimes _{\min} C \to B \otimes _{\min} C \to B/J \otimes _{\min} C \to 0$$
is exact for every C*-algebra $C$.
\end{theorem}
\noindent 
Consequently, $\mr C^*(\F_2)$ is not locally reflexive and $\mr C^*_r(\F_2)$ does not admit the LLP.

We will now show that the analogue of Theorem $\ref{wass}$ holds when $\F_2$ is replaced with a QSIN group $G$ containing a closed copy of $\F_2$. As an immediate consequence, we will have that $\mr C^*(G)$ is not locally reflexive and $\mr C^*_r(G)$ does not admit the LLP. Before proceeding to the proof, we note that a very simple argument achieves this result when $G$ is discrete.

\begin{remark}\label{local remark}
Let $G$ be a discrete group containing a copy of $\F_2$. Then  $\mr C^*(G)$ is not locally reflexive since local reflexivity passes to subspaces.

Next let $P$ be the orthogonal projection from $\ell^2(G)$ onto $\ell^2(\F_2)$ and observe that $\Phi\fn \mr C^*_r(G)\to \mr C^*_r(\F_2)$ defined by $\Phi(a)=PaP$ is a conditional expectation. Thus if $\phi\fn \mr C^*_r(\F_2)\to B$ is any ucp map into a C*-algebra $B$, then $\tilde{\phi}:=\phi\circ\Phi$ is a ucp map from $\mr C^*_r(G)$ to $B$ which extends $\phi$. This implies that if $\mr C^*_r(G)$ had the LLP, then so would $\mr C^*_r(\F_2)$. Since $\mr C^*_r(\F_2)$ does not have the LLP, we conclude that neither does $\mr C^*_r(G)$.
\end{remark}

\begin{lemma}\label{ind lemma}
	Let $G_1$ and $G_2$ be groups, and $H_1$ a closed subgroup of $G_1$. If $\pi$ is a representation of $H_1\times G_2$, then
	\begin{enumerate}
		\item[(a)] $(\Ind_{H_1\times G_2}^{G_1\times G_2}\pi)|_{G_1}$ is unitarily equivalent to $\Ind_{H_1}^{G_1} (\pi|_{H_1})$,
		\item[(b)] $(\Ind_{H_1\times G_2}^{G_1\times G_2}\pi)|_{G_2}$ is quasi-contained in $\pi|_{G_2}$.
	\end{enumerate}
\end{lemma}

\begin{proof}
	(a) Fix a quasi-invariant measure $d\dot g_1$ on $G_1/H_1$. This becomes a quasi-invariant measure on $(G_1\times G_2)/(H_1\times G_2)$ when we identify $(G_1\times G_2)/(H_1\times G_2)$ with $G_1/H_1$.  It is straightforward to verify that $U\fn \mc F(G_1,\pi|_{H_1})\to \mc F(G_1\times G_2,\pi)$ defined by
	$$ (Uf)(g_1,g_2)=\pi(e,g_2^{-1})f(g_2) $$
	is a surjective intertwining isometry.
	
	(b) Let $f\in \mr C_{\mr c}(G_1)$ and $\xi\in\Hi_\pi$. Then $x\fn G_1\times G_2\to \Hi_\pi$ defined by
	$$ x(g_1,g_2)=\pi(e,g_2^{-1})\int_{H_1}\delta_{H_1}^{G_1}(h_1^{-1})f(g_1h_1) \pi(h_1,e)\xi\,dh_1$$
	is in $\mc F(G,\Hi_\pi)$ by the identification of $\mc F(G_1,\pi|_{H_1})$ with $\mc F(G_1\times G_2,\pi)$ established above. A straightforward calculation shows that
	$$(\Ind_{H_1\times G_2}^{G_1\times G_2}\pi)_{x,x}(e,g_2)=\pi_{\xi,\eta}(e,g_2)$$
	where
	$$ \eta=\int_{G_1/H_1}\int_{H_1}\int_{H_1} \delta_{H_1}^{G_1}(hh')^{-1} \overline{f(g_1h)}f(g_1h')\pi(h^{-1}h')\xi\,dh\,dh'\,d\dot g_1.$$
	In particular $(\Ind_{H_1\times G_2}^{G_1,G_2}\pi)_{x,x}|_{G_2}\in \mr A_{\pi|_{G_2}}$. Since the linear span of functions with the form of $x$ (as $f$ ranges over $\mr C_{\mr c}(G_1)$ and $\xi$ over $\Hi_\pi$) is a dense subspace of $\overline{\mc F}(G_1\times G_2,\pi)$ (see \cite[Lemma 2.24]{kt}), it follows that $\mr A_{(\Ind\,\pi)|_{G_2}}\subset \mr A_{\pi|_{G_2}}$. Thus $(\Ind_{H_1\times G_2}^{G_1\times G_2}\pi )|_{G_2}$ is quasi-contained in $\pi|_{G_2}$.
\end{proof}

Recall that $\mr L^1(G\times G)=\mr L^1(G)\otimes_\gamma\mr L^1(G)$, where $\otimes_\gamma$ denotes the Banach space projective tensor product. As such, we may view C*-completions of tensor products of group C*-algebras of $G$ as being C*-completions of $\mr L^1(G\times G)$. This is used implicitly in the following theorem when we identify $*$-representations of tensor products of group C*-algebras of $G$ with representations of $G\times G$.

\begin{theorem}
	Let $G$ be a QSIN group which contains $\F_2$ as a closed subgroup. Then the sequence
	\begin{equation}\label{QSIN seq}
	0\to \mr C^*(G)\otimes_{\min}K\to \mr C^*(G)\otimes_{\min}\mr C^*(G)\to \mr C^*(G)\otimes_{\min} C^*_r(G) \to 0
	\end{equation}
	is not exact, where $K$ is the kernel of the canonical map $\mr C^*(G)\to \mr C_r^*(G)$.
\end{theorem}


\begin{proof}
	Wasserman's result above is equivalent to the statement that the quotient map
	$$\big(\mr C^*(\F_2)\otimes_{\min}\mr C^*(\F_2)\big)/\big(\mr C^*(\F_2)\otimes_{\min} J\big)\to \mr C^*(\F_2) \otimes_{\min}\mr C^*_r(\F_2)$$
	is not injective. Let $\sigma\fn \F_2\times\F_2\to \big(\mr C^*(\F_2)\otimes_{\min}\mr C^*(\F_2)\big)/\big(\mr C^*(\F_2)\otimes_{\min} J\big)$ be the canonical map, $\pi_{u}\fn \F_2\to \B(\Hi_u)$ be the universal representation of $\F_2$, and $\lambda_{\F_2}\fn \F_2\to \mc B(\ell^2(\F_2))$ the left regular representation of $\F_2$. Then $\pi_{u}\times\lambda_{\F_2}$ does not weakly contain $\sigma$. This will be used in our proof that sequence 2 is not exact.
	
	Consider the representation $\omega:=\Ind_{\F_2\times \F_2}^{G\times G}\sigma$. Observe that $\omega$ is weakly contained in $\Ind_{\F_2}^G\pi_u\times \Ind_{\F_2}^G\pi_u$ since $\sigma$ is weakly contained in $\Ind_{\F_2\times\F_2}^{G\times G}(\pi_u\times\pi_u)$ and $\Ind_{\F_2\times\F_2}^{G\times G}(\pi_u\times\pi_u)=\Ind_{\F_2}^G\pi_u\times \Ind_{\F_2}^G\pi_u$ (see \cite[Theorem 2.53]{kt}). So $\omega$ extends to a $*$-representation of $\mr C^*(G)\otimes_{\min}\mr C^*(G)$.
	Notice that $\omega|_{\{e\}\times G}$ is weakly contained in $\lambda_G$ since
	$$(\Ind_{\F_2\times \F_2}^{G\times G}\sigma)|_{\{e\}\times G}=\big(\Ind_{\F_2\times G}^{G\times G}(\Ind_{\F_2\times \F_2}^{\F_2\times G}\sigma)\big)|_{\{e\}\times G}$$
	is quasi-contained in $\Ind_{\F_2}^{G}(\sigma|_{\{e\}\times G})$ by Lemma \ref{ind lemma} and $\Ind_{\F_2}^{G}(\sigma|_{\{e\}\times G})$ is weakly equivalent to $\Ind_{\F_2}^G \lambda_{\F_2}=\lambda_G$, where $\lambda_G$ is the left regular representation of $G$. So the kernel of the representation $\omega$ when applied to $\mr C^*(G)\otimes_{\min} \mr C^*(G)$ contains $\mr C^*(G)\otimes_{\min} K$.
	If sequence \ref{QSIN seq} were exact, then $\omega$ would be weakly contained in $\pi_{u,G}\times \lambda_G$, where $\pi_{u,G}$ is the universal representation of $G$. Then $\omega|_{\F_2\times \F_2}$ would be weakly contained in $(\pi_{u,G}\times \lambda_G)|_{\F_2\times \F_2}$, which is weakly contained in $\pi_u\times \lambda_{\F_2}$ since $\pi_{u,G}|_{\F_2}$ is quasi-contained in $\pi_{u}$ and $\lambda_{G}|_{\F_2}$ is quasi-equivalent to $\lambda_{\F_2}$ by the Herz restriction theorem (see \cite{herz}). 
	This, however, is a contradiction since $\omega|_{\F_2\times\F_2}$ weakly contains $\sigma$ by Theorem \ref{main thm} and $\sigma$ is not weakly contained in $\omega|_{\F_2}$.
	Thus sequence 2 is not exact.
\end{proof}

\begin{corollary}
	If $G$ is a QSIN group which contains $\F_2$ as a closed subgroup, then $\mr C^*(G)$ is not locally reflexive and $\mr C^*_r(G)$ does not have the LLP.
\end{corollary}

\section{(Weak) Extensions of Fourier(-Stieltjes) spaces}

Let $G$ be a locally compact group and $H$ a closed subgroup of $G$. We say that Fourier space $\mr A_\pi$ of $H$ {\it extends} to a Fourier space of $G$ if there exists a Fourier space $\mr A_\sigma$ of $G$ such that $\mr A_\sigma|_{H}=\mr A_\pi$. Equivalently, $\mr A_\pi$ extends to a Fourier space of $G$ if and only if there exists a representation $\sigma$ of $G$ such that $\sigma|_H$ is quasi-equivalent to $G$ if and only if $\mr{VN}_\pi= \mr{VN}_{\sigma|_H}$ canonically. Similarly we say that a Fourier-Stieltjes space $\mr B_\pi$ of $H$ {\it weakly extends} to a Fourier-Stieltjes space of $G$ if there exists a Fourier-Stiejtes space $\mr B_\sigma$ of $G$ such that $\mr B_\sigma|_H$ is a weak* dense subspace of $\mr B_\pi$. Equivalently, $\mr B_\pi$ weakly extends to a Fourier-Stieltjes space of $G$ if and only if there exists a representation $\sigma$ of $G$ such that $\sigma|_H$ is weakly equivalent to $\pi$ if and only if $\mr C^*_{\pi}=\mr C^*_{\sigma|_H}$ canonically.

In this section we will look at when a Fourier(-Stieltjes) space $\mr A_\pi$ (resp. $\mr B_\pi$) of a closed subgroup $H$ of a locally compact group $G$ (weakly) extends to a Fourier(-Stieltjes) space of $G$. We will focus our attention when $\mr A_\pi$ (resp. $\mr B_\pi$) is  an ideal of the Fourier-Stieltjes algebra $\mr B(H)$. The majority of well studied Fourier and Fourier-Stieltjes spaces fall into this category. The special case of the extension problem when the Fourier space $\mr A_\pi$ is taken to be the Fourier algebra $\mr A(H)$ or the Fourier-Stieltjes algebra $\mr B(H)$ are already well studied. The Herz restriction theorem immediately implies that the Fourier algebra $\mr A(H)$ always extends to a Fourier space of $G$ for every locally compact group $G$, and a previously mentioned result of Cowling and Rodway gives that the Fourier-Stieltjes algebra $\mr B(H)$ extends to a Fourier space of $G$ whenever $G$ is a SIN group.

\begin{theorem}[Herz Restriction Theorem \cite{herz}]
	Let $H$ be a closed subgroup of a locally compact group $G$. Then $\mr A(G)|_H=\mr A(H)$.
\end{theorem}

\begin{theorem}[Cowling-Rodway \cite{cr}] \label{cr2}
	Let $H$ be a closed subgroup of a SIN group $G$. Then $\mr B(G)|_H=\mr B(H)$.
\end{theorem}

\noindent As a further example of a case which has been studied, Ghandehari has shown that the analogue of Theorem \ref{cr2} holds for the Rajchman algebra $\mr B_0(G)$ (see \cite{ghand}).

We observe that the weak analogue of Theorem \ref{cr2} for the class of QSIN groups is an immediate consequence of Theorem \ref{main thm}.

\begin{corollary}\label{weak CR}
	Let $H$ be a closed subgroup of a QSIN group $G$. Then $\mr B(G)|_H$ is weak* dense in $\mr B(H)$.
\end{corollary}

Let $G=\R\rtimes \R^+$ be the $ax+b$ group and $H$ be the subgroup $\R$. Then $G$ is a QSIN group by virtue of being amenable, but $\mr B(G)|_H=\mr A(\R)\oplus\C1_\R\subsetneq \mr B(\R)$ (see \cite{khalil}). Thus the conclusion of Corollary \ref{weak CR} cannot be replaced with $\mr B(G)|_H=\mr B(H)$. Further, there are also counterexamples to Corollary \ref{weak CR} when we drop the assumption that $G$ is QSIN (see \cite[Lemma 3]{bv}).

We now characterize when an arbitrary Fourier(-Stieltjes) space of a closed subgroup $H$ of a (Q)SIN group $G$ extends to a Fourier(-Stieltjes) space of $G$.

\begin{lemma}
	Let $G$ be a locally compact group, $H$ a closed subgroup of $G$, and $\pi$ a representation of $G$.
	\begin{itemize}
		\item[(a)] Suppose $G$ is a SIN  group and $\mr A_\pi$ is an ideal of $\mr B(H)$. If $\mr A_\pi$ extends to a Fourier space of $G$, then there exists a representation $\sigma$ of $G$ such that $\mr A_\sigma|_H=\mr A_{\pi}$ and $\mr A_\sigma$ is an ideal of $\mr B(G)$.
		\item[(b)] Suppose $G$ is a QSIN group and $\mr B_\pi$ is an ideal of $\mr B(H)$. If $\mr B_\pi$ weakly extends to a Fourier-Stieltjes space of $G$, then there exists a representation of $\sigma$ of $G$ such that $\mr B_\sigma|_H$ is weak* dense subspace of $\mr B_\pi$.
	\end{itemize}
\end{lemma}

\begin{proof}
	We will prove (b). The proof of (a) is similar but easier.
	
	Suppose that $\sigma'$ be a representation of $G$ such that $\mr B_{\sigma'}|_H$ is a weak* dense subspace of $\mr B_\pi$. Define
	$$ E=\mr A_{\sigma'}\cdot \mr B(G)=\{uv : u\in \mr A_{\sigma'}, v\in \mr B(G)\}.$$
	Then $E$ is a translation invariant ideal of $\mr B(G)$ and, so, the weak* closure of $E$ is equal to some Fourier-Stieltjes space $\mr B_\sigma$ of $G$ where $\mr B_\sigma$ is an ideal of $\mr B(G)$. Next observe that the weak* closure of $E|_H$ in $\mr B(H)$ is $\mr B_\pi$ since $\mr A_{\sigma'}|_H\subset E|_H\subset \mr B_{\pi}$. Thus $\mr B_{\sigma}|_H$ is a weak* dense subset of $\mr B_\pi$.
\end{proof}

\begin{prop}
	Let $G$ be a locally compact group, $H$ a closed subgroup of $G$, and $\pi$ a representation of $G$.
	\begin{enumerate}
		\item[(a)] Suppose $G$ is a SIN group and $\mr A_\pi$ is an ideal of $\mr B(G)$. Then $\mr A_{\Ind_H^G\pi}$ is the smallest Fourier space of $G$ which is an ideal of $\mr B(G)$ and whose restriction to $H$ contains $\mr A_\pi$.
		\item[(b)] Suppose $G$ is a QSIN group and $\mr B_\pi$ is an ideal of $\mr B(G)$. Then $\mr B_{\Ind_H^G\pi}$ is the smallest Fourier-Stieltjes space of $G$ which is an ideal of $\mr B(G)$ and whose restriction to $H$ contains $\mr B_\pi$ in its weak* closure in $\mr B(H)$.
	\end{enumerate}
\end{prop}

\begin{proof}
	We will show (b). The proof of (a) follows by replacing Theorem \ref{main thm} with Theorem \ref{CR} and weak containment by quasi-containment.
	
	By Theorem \ref{main thm}, $\pi$ is weakly contained in $(\mathrm{Ind}_H^G \pi)|_H$. So $\mr B_\pi\subset \mr B_{(\mathrm{Ind}_H^G \pi)|_H}$. Since $\mr A_{(\mathrm{Ind}_H^G \pi)|_H}\subset \mr B_{(\mathrm{Ind}_H^G \pi)}|_H$, we deduce that $\mr B_\pi \subset \overline{\mr B_{\mr{Ind}_H^G\pi}|_H}^{w^*}$.	Next suppose that $\sigma$ is a representation of $G$ such that $\mr B_\sigma$ is an ideal of $\mr B(G)$ and $\overline{(\mr B_\sigma|_H)}^{w^*}\supset \mr B_\pi$. Then $\sigma\otimes \Ind_H^G 1_H\prec \sigma$ since $\mr B_\sigma$ is an ideal of $\mr B(G)$. Further,
	$$\Ind_H^G \pi\prec\Ind_H^G (\sigma|_H)=\Ind_H^G (\sigma|_H\otimes 1_H)=\sigma\otimes \Ind_H^G 1_H  $$
	(see \cite[Theorem 2.58]{kt}) since $\pi\prec \sigma|_H$. Thus $\Ind_H^G\pi\prec \sigma$ or, equivalently, $\mr B_{\mathrm{Ind}_H^G\pi}\subset \mr B_\sigma$. Finally, $\mr B_{\Ind_H^G \pi}$ is an ideal of $\mr B(G)$ since $\sigma\otimes\Ind_H^G \pi=\mathrm{Ind}_H^G(\sigma|_H\otimes \pi)\prec \Ind_H^G \pi$ for every representation $\sigma$ of $G$.

\end{proof}

\begin{corollary}\label{extensions cor}
	Let $G$ be a locally compact group, $H$ a closed subgroup of $G$, and $\pi$ a representation of $H$.
	\begin{enumerate}
		\item[(a)] Suppose $G$ is a SIN group and $\mr A_\pi$ is an ideal of $\mr B(H)$. The following are equivalent.
		\begin{enumerate}
			\item[(i)] $\mr A_\pi$ extends to a Fourier space of $G$;
			\item[(ii)] $\mr A_{\Ind_H^G \pi}|_H\subset \mr A_\pi$;
			\item[(iii)] $\mr A_{\Ind_H^G \pi}|_H= \mr A_\pi$.
		\end{enumerate}
		\item[(b)] Suppose that $G$ is a QSIN group and $\mr B_\pi$ is an ideal of $\mr B(H)$. The following are equivalent.
		\begin{enumerate}
			\item[(i)] $\mr B_\pi$ weakly extends to a Fourier-Stieltjes space of $G$;
			\item[(ii)] $\mr B_{\Ind_H^G \pi}|_H\subset \mr B_\pi$;
			\item[(iii)] $\mr B_{(\Ind_H^G \pi)|_H}= \mr B_\pi$.
		\end{enumerate}
	\end{enumerate}
\end{corollary}

We note that Corollary \ref{extensions cor} need not hold when the condition that $\mr A_\pi$ is simply a subalgebra of $\mr B(H)$. Indeed, suppose that $H$ is a non-normal closed subgroup of a locally compact group $G$. Then $\mr A_{1_H}$ is subalgebra of $\mr B(G)$ which clearly extends to a Fourier space of $G$ despite the fact that $(\Ind_H^G 1_H)|_H$ is not quasi-contained in $1_H$.

We end this paper with an example which shows that not every Fourier(-Stieltjes) space of a closed subgroup $H$ of a (Q)SIN group $G$ (weakly) extends to Fourier(-Stieltjes) space of $G$.

\begin{example}
	Let $G_1$ be any noncompact locally compact group. Then $Z/2\Z$ acts on $G_1\times G_1$ by switching coordinates. Consider the group $G:=(G_1\times G_1)\rtimes (\Z/2\Z)$ with closed subgroup $H:=G_1\times G_1$.
	
	Define a subspace $E$ of the Fourier-Stietjes algebra $\mr B(H)$ by
	$$ E=\big\{u\in\mr B(H) : u|_{G_1\times \{e\}}\in \mr A(G_1\times \{e\})\big\}.$$
	Then $E$ is a norm closed translation invariant ideal of $\mr B(G)$, i.e, $E$ is a Fourier space which is an ideal of $\mr B(H)$. However $E$ does not extend to a Fourier space of $G$. Indeed, if $\mr A_\sigma$ is a Fourier space of $G$ such that $\mr A_\sigma|_H\supset E$, then $\mr A_\sigma|_{G_1\times \{e\}}=\mr B(G_1\times \{e\})\supsetneq \mr A(G_1\times \{e\})$ since $\mr A_\sigma$ is translation invariant and $u\times v\in E$ for all $u\in \mr A(G_1)$ and $v\in \mr B(G_1)$.
	
	The same analysis as above shows that if $G_1$ is any nonamenable group, and $G$ and $H$ are defined as above, then
	$$ E':= u\in \{u\in \mr B(H): u|_{G_1\times \{e\}}\in \mr B_r(G_1\times\{e\})\}$$
	(where $\mr B_r(H)$ denotes the reduced Fourier-Stieltjes algebra of $H$) is a Fourier-Stieltjes space of $H$ which is an ideal in $\mr B(H)$ but does not weakly extend to a Fourier-Stieltjes space of $G$. 
\end{example}

\end{document}